\documentclass{amsart}
\usepackage[utf8]{inputenc}

\usepackage{amsthm}
\usepackage{amsmath}
\usepackage{amsfonts}
\usepackage{amssymb}
\usepackage{fullpage}
\usepackage{xcolor}

\usepackage{hyperref,cleveref,color,verbatim}

\newtheorem{theorem}{Theorem}[section]
\newtheorem{prop}[theorem]{Proposition}
\newtheorem{lemma}[theorem]{Lemma}
\newtheorem{cor}[theorem]{Corollary}
\newtheorem{conj}[theorem]{Conjecture}

\newtheorem*{theorem*}{Theorem}

\theoremstyle{definition}
\newtheorem{remark}[theorem]{Remark}
\newtheorem{defi}[theorem]{Definition}

\newcommand{\fS}{\mathfrak{S}}
\newcommand{\R}{\mathbb{R}}
\newcommand{\Q}{\mathbb{Q}}

\newcommand\Def[1]{\emph{#1}}%
\newcommand\x{{x}}%
\newcommand\X{{X}}%
\newcommand\I{{I}}%
\renewcommand\a{{a}}
\renewcommand\v{{v}}%

\DeclareMathOperator*{\Diag}{Diag}
\DeclareMathOperator*{\diag}{diag}

\DeclareMathOperator*{\Adj}{Adj}
\DeclareMathOperator*{\SO}{SO}

\DeclareMathOperator{\conv}{\operatorname{conv}}

\newcommand*{\Sym}{\R^{n \times n}_{\mathrm{sym}}}

\newcommand{\De}{\operatorname{D}}

\title{Linear Principal Minor Polynomials: hyperbolic determinantal
inequalities and spectral containment}
\date{\today}

\author{Grigoriy Blekherman}
\address[Grigoriy Blekherman]{School of Mathematics, Georgia Institute of Technology, Atlanta, Georgia}
\email{greg@math.gatech.edu}
\author{Mario Kummer}
\address[Mario Kummer]{Technische Universit\"at Dresden, Germany}
\email{mario.kummer@tu-dresden.de}
\author{Raman Sanyal}
\address[Raman Sanyal]{Institut f\"ur Mathematik, Goethe-Universit\"at Frankfurt, Germany} 
\email{sanyal@math.uni-frankfurt.de}
\author{Kevin Shu}
\address[Kevin Shu]{School of Mathematics, Georgia Institute of Technology, Atlanta, Georgia}
\email{kshu8@gatech.edu}
\author{Shengding Sun}
\address[Shengding Sun]{School of Mathematics, Georgia Institute of Technology, Atlanta, Georgia}
\email{ssun313@gatech.edu}

\thanks{We would like to thank Santanu Dey for many insightful conversations. Grigoriy Blekherman, Kevin Shu and Shengding Sun were partially supported by NSF grant DMS-1901950. Mario Kummer was partially supported by DFG grant 421473641}

\begin{document}

\maketitle

\begin{abstract}
    A \emph{linear principal minor polynomial} or \emph{lpm polynomial} is a
    linear combination of principal minors of a symmetric matrix. By
    restricting to the diagonal, lpm polynomials are in bijection to
    multiaffine polynomials. We show that this establishes a one-to-one
    correspondence between homogeneous multiaffine stable polynomials and
    PSD-stable lpm polynomials. This yields new construction techniques for
    hyperbolic polynomials and allows us to generalize the well-known
    Fisher--Hadamard and Koteljanskii inequalities from determinants to
    PSD-stable lpm polynomials. We investigate the relationship between the
    associated hyperbolicity cones and conjecture a relationship
    between the eigenvalues of a symmetric matrix and the values of certain lpm
    polynomials evaluated at that matrix. We refer to this relationship as spectral containment.
\end{abstract}


\section{Introduction}\label{sec:intro_new}%
A homogeneous polynomial $p \in \R[\x] := \R[x_1,\dots,x_n]$ is called
\Def{hyperbolic} with respect to $\a \in \R^n$ if $p(\a) \neq 0$ and
$p_{\a}(t) := p(\v-t \a) \in \R[t]$ has only real roots for all $\v \in \R^n$.
The \emph{hyperbolicity cone} $H_a(p)$ of a polynomial $p$ hyperbolic with
respect to $\a\in\R^n$ is the set of all $\v\in\R^n$ such that $p(\v-t\a)$ has
only nonnegative roots.  Originally conceived in the context of partial
differential equations~\cite{Ga51}, hyperbolic polynomials were discovered to
yield deep results in (non-)linear algebra, combinatorics, and optimization;
see, for example,~\cite{MR3754960, borcea2008applications,
BB09stabilitypreserver, MR3098077,SanSau,  MR2738906}.

\newcommand\1{\mathbf{1}}%
A fundamental family of hyperbolic polynomials is given by the \Def{elementary symmetric polynomials}
\[
    e_k(\x) \ := \ \sum_{J} \prod_{i \in J}x_i\,,
\]
where $J$ ranges over all $k$-element subsets of $[n] := \{1,\dots,n\}$.  The
elementary symmetric polynomials are \Def{stable}: a multivariate polynomial
$p\in\R[\x]$ is \Def{stable} if for all complex numbers $z_1,\ldots,z_n$ lying
in the open upper half-plane, we have $p(z_1,\ldots,z_n)\neq0$. If $p$ is
homogeneous, then it is stable if and only if it is hyperbolic with respect to
all $\a \in \R^n_{>0}$, and we denote by $H(p) = H_{\1}(p)$ its
hyperbolicity cone with respect to the vector $\1 = (1,\dots,1)$.

Let $\X$ denote an $n\times n$ matrix of indeterminants, and for any $J \subseteq [n]$, we let $\X_J$ denote the principal submatrix of $\X$ indexed by $J$.
We can then define a polynomial
\[
    E_k(\X) \ := \ \sum_{J} \det(\X_J) \, ,
\]
where again $J$ ranges over all $k$-element subsets of $[n]$.
It turns out that these polynomials do not vanish on the \Def{Siegel upper
half-plane}, i.e., the set of all complex symmetric matrices with positive definite imaginary part. Such polynomials are called \Def{Dirichlet--G\r{a}rding}~\cite{harvey2009hyperbolic} or \Def{PSD-stable}~\cite{JorgensTheobald}. For a homogeneous polynomial $P$ this property is equivalent to being hyperbolic with respect to any positive definite matrix, and we denote by $H(P)$ its hyperbolicity cone (taken with respect to the identity matrix).
When the context is clear, we will simply refer to PSD-stable polynomials $P(\X)$ as stable polynomials.

The starting point of our paper is the observation that $E_k(\X)$ is closely
related to $e_k(x)$.
For instance, if $\X = \Diag(x_1, \dots, x_n)$ is the diagonal matrix with diagonal entries $X_{ii} = x_i$, then $E_k(\X) = e_k(x_1, \dots, x_n)$.
To generalize this observation, let $\Sym$ be the vector space of real symmetric $n \times n$-matrices and let $\R[\X]$ be the ring of polynomials on it, where we regard $\X$ as being an $n\times n$ matrix of indeterminants.
A polynomial $P(\X) \in \R[\X]$ is called a \Def{linear
principal minor polynomial} or \Def{lpm-polynomial} if $P(\X)$ is of the form
\[
    P(\X) \ = \ \sum_{J} c_J \det(\X_J) \, ,
\]
where $J$ ranges over all subsets of $[n]$. The first natural question we
pursue is what interesting properties are shared by a homogeneous lpm
polynomial $P(\X)$ and its diagonal restriction $p(\x)$. We show that $P(\X)$
is PSD-stable if and only if $p(\x)$ is stable. We obtain a similar result for
the related concept of Lorentzian polynomials. We prove these facts using the
theory of stability preservers \cite{weyl}.

Having established these basic facts we generalize classical determinantal
inequalities from linear algebra, such as the Hadamard--Fischer and
Koteljanskii inequality to the setting of stable lpm polynomials. This
generalizes the Hadamard-type inequalities for $k$-positive matrices obtained
in~\cite{Hadamard-kPos}.
Another
interesting consequence of the above results is that they give construction of
a new class of hyperbolic polynomials. Using lpm polynomials we construct a
hyperbolic cubic in 6 variables which has a Rayleigh difference that is not a
sum of squares. The previously smallest known example with $43$ variables was
contructed by Saunderson in \cite{soshyperbolic}. Finally, we study whether
the eigenvalue vector $\lambda$ of a matrix $X$ lying in the hyperbolicity
cone of a stable lpm polynomial $P(\X)$ lies in the hyperbolicity cone of
$p(x)$ and show how this is related to a potential generalization of the
classical Schur--Horn theorem \cite{schur23, horn54}. We now discuss our
results in detail.


\section{Our results in detail}\label{sec:results}
Our discussion of lpm polynomials can also be viewed from a different perspective. A
polynomial $p \in \R[\x] := \R[x_1,\dots,x_n]$ is \Def{multi-affine} if it
is a linear combination of square-free monomials $x^J = \prod_{j \in J} x_j$
for $J \subseteq [n]$.
We define a linear map $\Phi$ from the vector subspace of multi-affine polynomials in $x_1,\ldots,x_n$ to the vector space of lpm polynomials, which we call the \Def{minor lift map}, as follows. The minor lift of
\[
    p(\x) = \sum_{J \subseteq [n]} a_J \prod_{i \in J} x_i,
\]
is the polynomial $P = \Phi(p)$ given by
\[
    P(X) = \sum_{J \subseteq [n] } a_J \det(\X_J).
\]
We note that $\deg(\Phi(p))=\deg(p)$ and that $\Phi(p)$ is homogeneous if and
only if $p$ is homogeneous.  When it is unambiguous, we will use lower case
letters such as $p$ to denote homogeneous, multiaffine $p \in \R[x_1, \dots,
x_n]$, and use the corresponding upper case letters for the minor lift, so
that $P$ is equal to $\Phi(p)$.

\newcommand\PSD{\mathrm{PSD}}%
\subsection{Properties of the minor lift map and constructions}
Our first result is that the minor lift map sends stable polynomials to
PSD-stable polynomials. Stronger even, let us call a matrix $A$
\Def{$k$-locally PSD} if every principal $k\times k$-submatrix $A_J$ of $A$ is
positive semidefinite. The collection $\PSD_k$ of $k$-locally PSD matrices is
a closed convex cone and $\PSD_d \subset \PSD_{d-1} \subset \cdots \subset
\PSD_1$.

\begin{theorem}\label{thm:minor_lift_stable}
    Let $p$ be a homogeneous multiaffine polynomial of degree $k$. If $p$ is
    stable, then $P = \Phi(p)$ is hyperbolic with $\PSD_k \subseteq H(P)$. In
    particular, $P$ is PSD-stable.
\end{theorem}

For $A \in \Sym$, let $\pi(A) = (A_{11},A_{22},\dots,A_{nn})$ be the projection to
the diagonal. A first implication for the associated hyperbolicity cones is
as follows.
\begin{cor}\label{cor:diag}
    Let $p$ be a homogeneous multiaffine stable polynomial and $P = \Phi(p)$.
    If $A \in H(P)$, then $p(\pi(A)) \geq P(A)$ and $\pi(A) \in H(p)$.
\end{cor}

Using Theorem \ref{thm:minor_lift_stable}, we are able to construct new
interesting hyperbolic polynomials.  Given a hyperbolic polynomial $p$ and
points $\a,\v$ in the hyperbolicity cone of $p$, the \emph{Rayleigh
difference} $\Delta_{\v,\a}(p)=D_{\v}p \cdot D_{\a} p - p \cdot D_{\v} D_{\a}
p$ is a polynomial nonnegative on $\mathbb{R}^n$ \cite{interlacers}. If the
polynomial $\Delta_{\v,\a}(p)=D_{\v}p \cdot D_{\a} p - p \cdot D_{\v} D_{\a}
p$ is not a sum of squares, this has interesting implications for
determinantal representations as well as a hyperbolic certificate of
nonnegativity of $\Delta_{\v,\a}(p)$ which cannot be recovered by sums of
squares. Saunderson \cite{soshyperbolic} characterized all pairs $(d,n)$ for
which there exists such a hyperbolic polynomial $p\in\R[x_1,\ldots,x_n]$ of
degree $d$, except when $d=3$, where the smallest known example with a
Rayleigh difference that is not a sum of squares depends on 43 variables. We
are able to reduce the number of variables to 6. See
\Cref{sec:soshyperbolicity} for more details.

\begin{theorem}\label{thm:nonsoshyperbolic}
There exists an (explicit) degree-3 hyperbolic polynomial $p$ in $6$ variables
and vectors $\v,\a \in H(p)$ such that the Rayleigh difference
$\Delta_{\v,\a}(p)$ is not a sum-of-squares.  \end{theorem}

\subsection{Hyperbolic determinantal inequalities}
We generalize some well known theorems from linear algebra to the setting of
lpm polynomials. Note that the cone of positive semidefinite matrices is
precisely the hyperbolicity cone of $\det(\X)$, which is the minor lift of
$e_n(x) = x_1\cdots x_n$.  For our generalizations, we replace the determinant
by the minor lift of a homogeneous multiaffine stable polynomial, and the cone
of positive semidefinite matrices by the hyperbolicity cone of the minor lift. 

Hadamard's inequality is a classical result comparing the determinant of any positive semidefinite matrix with the product of its diagonal entries. 

\begin{theorem*}[Hadamard's inequality]
	Let $A$ be a $n\times n$ positive semidefinite matrix, then $\det (A)\le \prod_{i=1}^n A_{ii}$. 
\end{theorem*}

An equivalent statement of this inequality is as follows: if $V$ is any, not
necessarily symmetric, real $n\times n$-matrix with columns $\v_1,\dots,
\v_n$, then $\det(V) \le \prod_{i=1}^n \|\v_i\|_2$. This yields a
geometric interpretation, since the absolute value of determinant is the
volume of an $n$-dimensional parallelepiped with edges $\v_1,\dots,\v_n$.

Fischer's inequality generalizes Hadamard's inequality, and relates the
determinant of a positive semidefinite matrix to its principal minors. Let
$\Pi = \{S_1, \dots, S_m\}$ be a partition of the set $[n]$ into $m$ disjoint
subsets.  Given such a partition, we write $i \sim j$ if $i,j \in S_k$ 
for some $k=1,\dots,m$.  Let $\mathcal{D}_{\Pi}$ be the vector space of
symmetric matrices that are \Def{block  diagonal} with respect to $\Pi$
\[
    \mathcal{D}_{\Pi} = \{A \in \Sym : A_{ij} = 0 \textnormal{ if } i \not \sim j \}.
\]
Let $\pi_{\Pi}$ be the orthogonal projection from $\Sym$ onto the subspace $\mathcal{D}_{\Pi}$. 

\begin{theorem*}[Fischer's inequality]
	Let $A$ be a positive semidefinite matrix. Then
	\[
        \det (\pi_{\Pi}(A))  \ \geq \ \det (A)  \, .
    \]
	
	\end{theorem*}
\noindent Observe that Hadamard inequality is simply Fischer's inequality
with partition $\Pi = \{ \{1\},\dots,\{n\}\}$. We now give a hyperbolic
generalization of Fischer-Hadamard inequality. 
For $P = \Phi(e_k)$, our hyperbolic Hadamard inequality was obtained
in~\cite{Hadamard-kPos}.

\begin{theorem}[Hyperbolic Fischer--Hadamard inequality]\label{thm:hadamard-fischer}
    Let $P$ be a homogeneous PSD-stable lpm-polynomial and $\Pi$ a partition.
    Then
    \[
        P(\pi_{\Pi}(A))  \ \ge \  P(A)
    \]
    holds for all $A \in H(P)$.
\end{theorem}

The classical Fischer--Hadamard inequality is a consequence of a more general
inequality known as Koteljanskii's inequality, which handles the case of
overlapping blocks \cite{Koteljanskii}.

\begin{theorem*}[Koteljanski's inequality]
Let $S$ and $T$ be two subsets of $[n]$ and $A$ be a positive semidefinite $n \times n$ matrix. Then
\[
    \det(A_S) \det(A_T) \ \ge \ \det(A_{S \cup T}) \det(A_{S \cap T}) \, .
\]
\end{theorem*}

While we were not able to generalize Koteljanskii's inequality in a way that implies the hyperbolic Fischer--Hadamard inequality, we found a hyperbolic generalization of Koteljanskii's inequality, which uses a different interpretation of what it means to take a minor of a matrix.

\begin{defi}\label{def:restriction}
    Given a degree $k$ homogeneous lpm polynomial $P$ and $T\subseteq [n]$
    with $|T|\ge n-k$, we define the
    \Def{restriction} 
    \[
        P|_T \ := \ \Bigl(\prod_{i\in[n]\setminus T}\frac{\partial }{\partial
        X_{ii}}\Bigr)P \, ,
    \]
   where we take partial derivative with respect to diagonal variables not in
    $T$. 
\end{defi}
With this definition we can state the hyperbolic Koteljanskii inequality, which is related to the negative lattice condition in \cite{BBL09negativedependence}:

\begin{theorem}[Hyperbolic Koteljanskii inequality]\label{thm:kota}
    Let $P$ be a homogeneous PSD-stable lpm-polynomial and 
$S,T\subseteq [n]$. Then 
    \[
        P|_S (A) P|_T (A) \ \ge \ P|_{S\cup T} (A) P|_{S\cap T} (A)
    \]
    holds for all $A\in H(P)$.
\end{theorem}

\subsection{Spectral containment property} 
If $A$ is an $n\times n$ symmetric matrix,
we say that $\lambda \in \R^n$ is an eigenvalue vector of $A$ if the entries of $\lambda$
are precisely the eigenvalues of $A$ with appropriate multiplicities.
Note that the set of eigenvalue vectors of a symmetric matrix $A$ are invariant under permutations.

We recall the example of the $k$-th elementary symmetric polynomial $e_k(\x)$ and its minor lift $E_k(\X)$ from the introduction.
It is well-known that $E_k(A)=e_k(\lambda)$ where $\lambda$ is an eigenvalue vector of $A$.
In particular, it follows that $A\in H(P)$
implies that $\lambda \in H(p)$. Notice that since $e_k$ is invariant under
permutations of coordinates, the order in which we list the eigenvalues of $A$
in $\lambda(A)$ does not matter. This motivates the following definition.

\begin{defi}\label{def:eigenvalue_containment}
    A homogeneous multiaffine stable polynomial $p \in \R[x_1,\dots,x_n]$ has
    the \Def{spectral containment property} if for any $A \in H(P) \subset
    \Sym$, there is an eigenvalue vector $\lambda \in \R^n$ of $A$
    such that $\lambda \in H(p)$.
\end{defi}
\begin{remark}
    We could make a stronger requirement in \Cref{def:eigenvalue_containment} that for all $A \in H(P)$, \emph{all} eigenvalue vectors of $A$ lie in $H(p)$, seems to be too restrictive; we do not have any examples of polynomials besides the elementary symmetric polynomials with this stronger property.
\end{remark}

We now give a number of polynomials which have the spectral containment property:
\begin{theorem}\label{thm:eigenvalue_containment}
    The following classes of polynomials have the spectral containment property:
    \begin{enumerate}
        \item The elementary symmetric polynomials $e_1,\dots,e_n$.
        \item For any $n \ge k \ge d$, and any $|\varepsilon|$ sufficiently small, $e_d(x_1, \dots, x_n) + \varepsilon e_d(x_1, \dots, x_k)$.
        \item Stable linear polynomials.
        \item Any degree $n-1$ stable polynomial that interlaces $e_{n-2}$.
        \item $e_2(x_1, x_2, x_3, x_4) - \varepsilon(x_1x_2 + x_1x_3)$ for $\varepsilon$ sufficiently small.
    \end{enumerate}
    Moreover, if $p$ has the spectral containment property, and $x_0$ is a variable not used in $p$, then $x_0p$ has the spectral containment property.
\end{theorem}
While this property may seem mysterious, we conjecture that it is in fact ubiquitous:
\begin{conj}\label{conj:eigenvalue_containment_strong}
    Every homogeneous multaffine stable polynomial has the spectral containment property.
\end{conj}

If \Cref{conj:eigenvalue_containment_strong} is true, then
\Cref{thm:minor_lift_stable} implies that for every $k$-locally PSD matrix $A$
and homogeneous multiaffine stable polynomial $p$, some eigenvalue vector of
$A$ is contained in $H(p)$.  This may seem like a very strong condition on
the eigenvalues of $A$, but as we show below it is equivalent to the fact that every
eigenvalue vector of $A$ is contained in $H(e_k)$, which we already observed
above.  Let $\fS_n$ denote the symmetric group on $n$ letters and let it act
on $\mathbb{R}^n$ by permuting coordinates. 

\begin{theorem}\label{thm:permutation}
 Let $e_k\in\R[\x]$ be the elementary symmetric polynomial of degree $k$ and $h\in\R[\x]$ be a nonzero homogeneous multiaffine stable polynomial of degree $k$.
 If $\v \in H(e_k)$, then there exists a permutation $\tau \in \fS_n$ such that $\tau(\v) \in H(h)$.
\end{theorem}
In \Cref{sec:schurhornprop} we will also show that
\Cref{conj:eigenvalue_containment_strong} would be implied in many cases by
another conjecture generalizing the classical Schur--Horn Theorem.

\section{The Minor Lift Map and Stability Preservers}\label{sec:minor_lift}
Our goal in this section is to prove \Cref{thm:minor_lift_stable}. 
We first explain how to construct the minor lift map via partial derivatives of the determinant. 
Let $p\in\R[\x]$ be a multiaffine polynomial. The \emph{dual} of $p$ is
\begin{equation}\label{eqn:dual}
    p^*(x) := x_1\cdot x_2\cdots x_n\cdot p\Bigl(\frac{1}{x_1}, \frac{1}{x_2},
    \dots, \frac{1}{x_n}\Bigr) .
\end{equation}
For any polynomial $p \in \R[x_1, \dots, x_n]$, we consider the differential
operator $p^*\left(\frac{\partial}{\partial X_{11}}, \frac{\partial}{\partial
X_{22}}, \dots, \frac{\partial}{\partial X_{nn}}\right)$.  For instance, if
$p= x^S = \prod_{i \in S}x_i$ is a monomial, then the associated differential
operator is $\prod_{i \notin S}\frac{\partial}{\partial X_{ii}}$.
Applying
the differential operator associated to $x^S$ to $\det(X)$
yields
\[
    \Bigl(\prod_{i \notin S}\frac{\partial}{\partial X_{ii}}\Bigr) \det(\X) = \det(X_{S}).
\]
By linearity, we then obtain that
\[
    P = \left(p^*\left(\frac{\partial}{\partial X_{11}},
    \frac{\partial}{\partial X_{22}}, \dots, \frac{\partial}{\partial
    X_{nn}}\right)\right) \det(X) \, .
\]

This formulation of the minor lift map will allow us to easily apply the theory of stability preservers.

\begin{remark}\label{rmk:minor_lift_dual}
    The minor lift operation interacts nicely with dualization.  If $p$ is a
    multiaffine polynomial, then
    \[
        \Phi(p^*)|_X = \det(X)\cdot\Phi(p)|_{X^{-1}}.
    \]
    Here, $\cdot|_X$ denotes the evaluation of a polynomial at $X$.

    This result follows directly from the Jacobi complementary minors
    identity, found in \cite{MR1411115}, which states that $\det(X|_{S^c}) =
    \det(X^{-1}|_S) \det(X)$.
    This is a matrix analogue of~\eqref{eqn:dual}.
\end{remark}

Before we go on, we need the following facts about hyperbolicity cones that can be found in \cite{MR2738906}.

\begin{lemma}\label{lem:hyp1}
    Let $p\in\R[\x]$ be a homogeneous polynomial and $K\subset\R^n$ a cone.
    The following are equivalent:
    \begin{enumerate}
        \item $p$ is hyperbolic with respect to all $\a\in K$, and
        \item $p(\v+\textnormal{i}\a)\neq0$  for all $\v\in\R^n$ and $\a\in K$.
    \end{enumerate}
\end{lemma}

\begin{lemma}\label{lem:hyp2}
 Let $p\in\R[\x]$ be hyperbolic with respect to $\a\in\R^n$. Then $p$ is hyperbolic with respect to every point in the connected component of $\{\v\in\R^n:\, p(\v)\neq0\}$ that contains $\a$.
\end{lemma}

Our first step is the following observation:

\begin{lemma}\label{lem:psd_stable}
 Let $P\in\R[\X]$ be a homogeneous polynomial. Then $P$ is PSD-stable if and only if the following two conditions hold:
\begin{enumerate}
    \item $P(A)\neq0$ for all positive definite matrices $A$;
    \item $P(\Diag(x_1, \dots, x_n) + M)\in\R[\x]$ is stable for every real symmetric matrix $M$.
\end{enumerate}
\end{lemma}

\begin{proof}
    First assume that $P$ is PSD-stable and let $A$ be a positive definite
    matrix. By definition we have $P(\mathrm{i} A)\neq0$. Since $P$ is
    homogeneous, this implies that $P(A)\neq0$. Further let
    $z_i=a_i+ \mathrm{i} b_i$ in the upper half-plane. Then $P(\Diag(z_1,
    \dots, z_n) + M)$ is nonzero for any real symmetric matrix $M$, since
    $\Diag(b_1, \dots, b_n)$ is a positive definite matrix.

    For the other direction we first observe that condition (2) implies that
    $P$ is hyperbolic with respect to the identity matrix. Indeed, the
    univariate polynomial $P(tI+M)$ is stable and thus real-rooted for every
    real symmetric matrix $M$. Now condition (1) together with Lemmas
    \ref{lem:hyp1} and \ref{lem:hyp2} imply the claim.
\end{proof}

\begin{proof}[Proof of \Cref{thm:minor_lift_stable}]
    Let $p\in\R[\x]$ be multiaffine, homogeneous and stable. Then by
    \cite[Thm.~6.1]{halfplane} all nonzero coefficients of $p$ have the same
    sign. Without loss of generality assume that all are positive. Then
    $P=\Phi(p)$ is clearly positive on positive definite matrices since the
    minors of a positive definite matrix are positive. Thus by
    \Cref{lem:psd_stable}, it remains to show that $$P(\Diag(x_1, \dots, x_n)
    + M)=\left(p^*\left(\frac{\partial}{\partial x_{1}},
    \frac{\partial}{\partial x_{2}}, \dots, \frac{\partial}{\partial
    x_{n}}\right)\right) \det(\Diag(x_1, \dots, x_n) + M)$$ is stable for
    every real symmetric matrix $M$. The polynomial $\det(\Diag(x_1, \dots,
    x_n) + M)$ is stable as well as $p^*$ by \cite[Prop.~4.2]{halfplane}. Thus
    the polynomial $P(\Diag(x_1, \dots, x_n) + M)$ is also stable by
    \cite[Thm.~1.3]{weyl}.

    Let $A \in \PSD_k \subseteq \Sym$ be $k$-locally PSD. Then for every
    $k$-subset $S \subseteq [n]$, we have $\det((A+t\I)|_S) > 0$ for all $t >
    0$. Hence, if $p$ has degree $k$ with all coefficients positive, then
    $P(A-t\I) > 0$ for all $t < 0$ and hence all roots are non-negative. This
    implies that $A \in H(P)$.
\end{proof}

\begin{remark}
    Given a multiaffine homogeneous stable polynomial
    $p\in\R[x_1,\ldots,x_n]$, the minor lift map gives a hyperbolic polynomial
    $P$ in the entries of a symmetric $n\times n$ matrix whose restriction to
    the diagonal equals to $p$. Such polynomials can also be constructed for
    stable polynomials that are not necessarily multiaffine. Since we are
    mainly interested in multiaffine polynomials, we only briefly sketch one
    possible such construction. To a stable homogeneous polynomial
    $p\in\R[x_1,\ldots,x_n]$ one can find a multiaffine stable polynomial
    $q\in\R[z_{11},\ldots,z_{1d_1},\ldots, z_{nd_n}]$ such that we can recover
    $p$ from $q$ by substituting each variable $z_{ij}$ by $x_i$, see \cite[\S
    2.5]{halfplane}. This polynomial $q$ is called a \emph{polarization} of
    $p$. If we restrict the minor lift of $q$ to suitable block-diagonal
    matrices, we obtain a hyperbolic polynomial with the desired properties
    for $p$.
\end{remark}

\begin{remark}
    Using \cite[Thm.~3.2]{branden2020lorentzian} one can show that the
    analogous statement to \Cref{thm:minor_lift_stable} for \emph{Lorentzian
    polynomials}, a recent generalization of stable polynomials,  holds as
    well.
\end{remark}

\section{Hyperbolic Hadamard-Fischer Inequality}\label{sec:hadamard-fischer}
Our goal in this section is to prove \Cref{thm:hadamard-fischer}.  We start by
making some general observations about supporting hyperplanes of the
hyperbolicity cone:
\begin{lemma}\label{lem:dual}
  Let $p\in\R[x]$ be hyperbolic with respect to $a\in\R^n$ and $H_a(p)$ the corresponding hyperbolicity cone. Assume that $p(a)>0$ and that $p$ is reduced in the sense that all its irreducible factors are coprime. Then we have the following:
  \begin{enumerate}
    \item For all $v\in H_a(p)$ the linear form $L_v=\langle\nabla p(v), x\rangle$ is nonnegative on $H_a(p)$.
    \item If $v\in\partial H_a(p)$, then $L_v(v)=0$.
    \item If $b\not\in H_a(p)$, then there exists $v\in\partial H_a(p)$ such that $L_v(b)<0$.
  \end{enumerate}
\end{lemma}

\begin{proof}
    Part (2) is just Euler's identity since $p$ vanishes on $\partial H_a(p)$. 
    If $\nabla p(v) = 0$, then (1) is trivial. Otherwise, the hyperplane $\{ x
    : \langle\nabla p(v), x\rangle = 0 \}$ is
    tangent to $\partial H_a(p)$ at $v$. Since $H_a(p)$ is convex and $p$
    positive on the interior of $H_a(p)$, this implies that $L_v(x) \ge 0$ for
    all $x \in H_a(p)$.
    In order to prove (3), we first note that by our assumption on $p$, the
    set of points $c\in\partial H_a(p)$ where $\nabla p(c) = 0$ is nowhere
    dense. Thus if $b\not\in H_a(p)$, then there is a point $e$ in the
    interior of $H_a(p)$ such that the line segment $[e,b]$ intersects
    $\partial H_a(p)$ in a smooth point $v$. Since $L_v(e)>0$ and $L_v(v)=0$,
    we have $L_v(b)<0$.
\end{proof}

We now apply the above observations to lpm polynomials. Recall that for a
partition $\Pi = \{S_1, \dots, S_m\}$ of $[n]$, we denote by
$\mathcal{D}_{\Pi}$ the vector space of block diagonal symmetric matrices with
blocks given by $\Pi$ and $\pi_{\Pi}$ is the orthogonal projection of $\Sym$
onto the subspace $\mathcal{D}_{\Pi}$. Further recall that we write $a\sim b$
for $a,b \in [n]$ if $a, b \in S_k$ for some $k=1,\dots,m$.

\begin{lemma}\label{lem:offblock}
    Fix a partition $\Pi=\{S_1, \dots, S_m\}$ of $[n]$ and let $B\subseteq
    [n]$ be any subset. Then for any $\sigma\in \fS_B$, we have $|\{b\in B |
    b\not\sim \sigma(b) \}|\ne 1$. 
\end{lemma}

\begin{proof}
    For $b \in B$, consider the orbit $b, \sigma(b),\sigma^2(b), \dots,
    \sigma^{t-1}(b), \sigma^t(b) = b$. If $b \in S_k$ but the orbit is not
    fully contained in $S_k$, then there are $0 \le r < s < t$ such that
    $\sigma^r(b),\sigma^{s+1}(b) \in S_k$ but $\sigma^{r+1}(b),\sigma^{s}(b)
    \not\in S_k$.
\end{proof}

\begin{lemma}\label{lem:derioff}
    Let $P$ be an lpm polynomial. If $A \in \mathcal{D}_{\Pi}$, then $\nabla P(A) \in \mathcal{D}_{\Pi}$.
\end{lemma}
\begin{proof}
    Since $P$ is a sum of terms of the form $a_B\det(X_B)$ with $B \subseteq
    [n]$, it suffices to prove the claim for $P=\det(X_B)$.
    In that case, this is equivalent to saying that if $A \in
    \mathcal{D}_{\Pi}$ and $i\not\sim j$, then
    \[
        \Bigl(\frac{\partial}{\partial X_{ij}}\det(X_B)\Bigr)(A) = 0.
    \]
    Now $\det X_B=\sum_{\sigma\in \fS_B} \mathrm{sgn}(\sigma) \prod_{i\in B}
    X_{i,\sigma(i)}$ and Lemma~\ref{lem:offblock} applied to each term yields
    the claim.
\end{proof}

The preceding lemma allows us to show that the hyperbolicity cone of a
hyperbolic lpm polynomial is closed under projections onto
$\mathcal{D}_{\Pi}$. 

\begin{lemma}\label{lem:lpm_projection}
    Let $P$ be a homogeous PSD-stable lpm polynomial. If $A \in H(P)$, then
    $\pi_{\Pi}(A) \in H(P)$.
\end{lemma}

\begin{proof}
    Let $P_{\Pi}$ be the restriction of the polynomial $P$ to
    $\mathcal{D}_{\Pi}$, i.e. $P_{\Pi} = P\circ\iota$ where $\iota :
    \mathcal{D}_{\Pi} \rightarrow \Sym$ is the inclusion map. We have $H(P)
    \cap \mathcal{D}_{\Pi} = H(P_{\Pi})$. For $A\in H(P)$ we thus have to show
    that $\pi_{\Pi}(A) \in H(P_\Pi)$. By \Cref{lem:dual} this is equivalent to
    $\langle\nabla P_\Pi(B), \pi_\Pi(A)\rangle\ge0$ for all $B\in H(P_\Pi)$.
    But by the previous lemma we have $\langle\nabla P_\Pi(B),
    \pi_\Pi(A)\rangle=\langle\nabla P(B), A\rangle$ which is nonnegative by
    \Cref{lem:dual} since $A \in H(P)$.
\end{proof}

We are now able to show the hyperbolic Fischer--Hadamard inequality. Our proof technique is inspired by the proof of \cite[Thm.~5]{Garding59ineq}. 

\begin{proof}[Proof of \Cref{thm:hadamard-fischer}]
Without loss of generality, we can assume that $P(I)>0$. If $A$ is on the boundary of $H(P)$, then $P(A)=0$ and we are done since $\pi_{\Pi}(A)\in H(P)$ implies $P(\pi_{\Pi}(A)) \geq 0$.
Therefore, we may assume that $A$ is in the interior of $H(P)$.
In this case, let $\epsilon>0$ be sufficiently small such that $A-\epsilon I\in H(P)$, then $\pi_\Pi(A)-\epsilon I=\pi_\Pi(A-\epsilon I)$ is also in $H(P)$.
This shows that $\pi_{\Pi}(A)$ is in the interior of $H(P)$ and $P(\pi_{\Pi}(A))>0$.
Thus $P$ is hyperbolic with respect to $A$ and $q(t) = P(tA+\pi_{\Pi}(A))\in\R[t]$ is real rooted with negative roots.
Let $d$ be the degree of $q(t)$ degree and $-\lambda_1,\ldots,-\lambda_d$ the roots of $q(t)$ with each $\lambda_i>0$.
We consider the coefficients of $t$ in $q(t)$: 

\begin{itemize}
    \item[--] The coefficient of $t^d$ is $P(A)$. 
    \item[--] The coefficient of $t$ is $dP(\pi_{\Pi}(A))$, since $\frac{d}{dt}q(t)|_{t=0} = \langle \nabla P|_{\pi_{\Pi}(A)}, A\rangle$, and by \Cref{lem:derioff}, $\langle \nabla P|_{\pi_{\Pi}(A)}, A\rangle = \langle \nabla P|_{\pi_{\Pi}(A)}, \pi_{\Pi}(A)\rangle = dP(\pi_{\Pi}(A))$. This last equality is due to Euler's identity.
    \item[--] The constant coefficient is $P(\pi_{\Pi}(A))$. 
\end{itemize}

Thus we have $e_{d-1}(\lambda)=\frac{dP(\pi_{\Pi}(A))}{P(A)}$, and $e_d(\lambda)=\lambda_1\cdots\lambda_d=\frac{P(\pi_{\Pi}(A))}{P(A)}$. Since all $\lambda_i$ are positive, from the AM-GM inequality we have 
$$\frac{P(\pi_{\Pi}(A))}{P(A)}=\frac{e_{d-1}(\lambda)}{d}\geq (\lambda_1\cdots\lambda_d)^{\frac{d-1}{d}}=\left(\frac{P(\pi_{\Pi}(A))}{P(A)}\right)^{\frac{d-1}{d}}.$$This proves the claim.
\end{proof}

When $P(X)=\det X$, then $H(P)$ is the cone of positive semidefinite matrices and our theorem implies the well-known Fischer's inequality: 

\begin{cor}[Fischer's inequality]
If $A$ is positive semidefinite, then $\det \pi_{\Pi}(A)\geq \det A$. 
\end{cor}

\begin{remark}
The usual statement of Fischer's inequality corresponds to the case of two blocks. This is equivalent to our multi-block version since principal submatrices of a positive semidefinite matrix are also positive semidefinite. 
\end{remark}

In the case, where $\Pi=\{\{1\},...,\{n\}\}$, 
\Cref{thm:hadamard-fischer} and \Cref{lem:lpm_projection} imply
\Cref{cor:diag}.
We also get the following strengthening of \Cref{thm:hadamard-fischer}.

\begin{cor}
\label{thm:incfh}
Let $P$ be a homogeneous and PSD-stable lpm-polynomial. If $A \in H(P)$, then the polynomial $P((1-t)A + t\pi_{\Pi}(A))$ is monotonically increasing for $t \in [0,1]$.
\end{cor}

\begin{proof}
    The polynomial $q(t) = P(tX + (\pi_{\Pi}(X) - X))$  is real rooted, and 
    \[
        P((1-t) X + t\pi_{\Pi}(X)) = q^*(t)
    \]
    so that $q^*(t)$ is real rooted.     
    Because both $A$ and $\pi_{\Pi}(A)$ are in $H(P)$, we have $q^*(t) \ge 0$ for $t \in [0,1]$. Hence, by interlacing $\frac{d}{dt} q^*(t)$ has at most one root in the interval $[0,1]$. If there were a root of $\frac{d}{dt} q^*(t)$ in the interval $[0, 1)$, then at such a root $q^*(t)$ would have a maximum on this interval. This is a contradiction to the fact that $q^*(t)$ is maximized at $t = 1$ by \Cref{thm:hadamard-fischer}.    
    Hence, we have that $\frac{d}{dt}q^*(t)$ must in fact be nonnegative on this interval.
\end{proof}

\section{Hyperbolic Koteljanskii inequality}\label{sec:koteljanskii}
Koteljanskii's inequality \cite{Koteljanskii} states that for any $n\times n$
positive semidefinite matrix $A$ and $S,T\subset [n]$, $\det A_S \det A_T\ge
\det A_{S\cap T} \det A_{S\cup T}$. This is a generalization of the
Hadamard--Fischer inequality. Later this inequality was proven to hold for
other classes of (possibly non-symmetric) matrices \cite{JN14koteljanskii}.  
In this section we prove \Cref{thm:kota}, a generalization of Koteljanskii's
inequality, where the determinant can be replaced by a PSD-stable
lpm polynomial. First we need the hyperbolic counterpart of the fact that
principal submatrices of a positive semidefinite matrix are again positive
semidefinite, and hence have nonnegative determinant. For this we use 
Renegar derivatives~\cite{Renegar06derivative}. 

\begin{theorem}\label{thm:Renegar}
	Let $p$ be a polynomial, hyperbolic with respect to $v$. Let $D_v p$ denote the directional derivative of $p$ in direction $v$. Then $D_v p$ is also hyperbolic with respect to $v$. Furthermore, their hyperbolicity cones satisfy $H_v(p)\subseteq H_v(D_v p)$. 
\end{theorem}

Recall from \Cref{def:restriction} that $P|_T=(\prod_{i\in[n]\setminus T}\frac{\partial }{\partial X_{ii}})P$. Then we have:

\begin{cor}\label{cor:subdet}
	Let $P$ be a homogeneous PSD-stable lpm polynomial of degree $k$ and $A\in H(P)$. Let $T\subseteq [n]$ with $|T|\ge n-k$. Then $P|_T$ is PSD-stable as well and $A\in H(P|_T)$. 
\end{cor}

Now we use the result from \cite{BBL09negativedependence} on negative dependence. For any polynomial $p\in\R[x]$ and index set $S\subseteq [n]$ we denote $\partial^S p=(\prod_{i\in S}\frac{\partial }{\partial x_{i}})p$. 

\begin{theorem}[{\cite[Sect.~2.1 and
    Thm.~4.9]{BBL09negativedependence}}]\label{thm:NLC}
    Let $p$ be a multiaffine stable polynomial with nonnegative coefficients.
    Then $p$ satisfies the nonnegative lattice condition: for all
    $S,T\subseteq [n]$
    \[
        \partial^S p(0) \partial^T p(0) \ \ge \ \partial^{S\cup T}
    p(0)\partial^{S\cap T}p(0) \, . 
\]
\end{theorem}

This theorem directly implies the generalization of Koteljanskii's inequality. 

\begin{proof}[Proof of \Cref{thm:kota}]
Without loss of generality assume that $P(I)>0$.
	Let $P_A(x)=P(A+\Diag(x))\in \R[x_1,...,x_n]$. It is clear that $P_A$ is multiaffine and $\partial^S P_A(0)=P|_S(A)$ for all $S\subseteq [n]$. It follows from Corollary \ref{cor:subdet} that $P_A$ is stable and has nonnegative coefficients. Thus by Theorem \ref{thm:NLC} it satisfies the nonnegative lattice condition, i.e. for all $S,T\subseteq [n],\partial^S P_A(0) \partial^T P_A(0)\ge \partial^{S\cup T} P_A(0)\partial^{S\cap T}P_A(0)$. This completes the proof.
\end{proof}

\section{Hyperbolic polynomials and sums of squares}\label{sec:soshyperbolicity} 
Let $p\in\R[x]$ be hyperbolic with respect to $v \in \R^n$ and $a,b \in H_v(p)$.  Then the mixed derivative
\[
    \Delta_{a,b}(p)=\De_ap\cdot\De_bp-p\cdot\De_a\De_bp
\] 
is globally nonnegative by Theorem 3.1 in \cite{interlacers}.  If
some power $p^r$ has a definite symmetric determinantal representation,
i.e., can be written as 
\[
    p^r=\det(x_1A_1+\cdots + x_n A_n) 
\]
for some real symmetric (or complex hermitian) matrices $A_1,\dots,A_n$ with
$v_1A_1+\ldots+v_nA_n$ positive definite, then $\Delta_{a,b}(p)$ is even a sum
of squares \cite[Cor.~4.3]{interlacers}.  Therefore, any instance where
$\Delta_{a,b}(p)$ is not a sum of squares gives an example of a hyperbolic
polynomial none of whose powers has a definite symmetric determinantal
representation.  Another source of interest in such examples comes from the
point of view taken in \cite{soshyperbolic}, as these give rise to families of
polynomials that are not sums of squares but whose nonnegativity can be
certified via hyperbolic programming.  Saunderson \cite{soshyperbolic}
characterized all pairs $(d,n)$ for which there exists such a hyperbolic
polynomial $p\in\R[x]=\R[x_1,\ldots,x_n]$ of degree $d$, except when $d=3$.
In this section we will construct an explicit hyperbolic cubic $p$ in $6$
variables for which there are two points $a,b$ in the hyperbolicity cone such
that $\Delta_{a,b}(p)$ is not a sum of squares.

\begin{remark}
    If there are two points $a,b$ in the closed hyperbolicity cone of $p$ such that $\Delta_{a,b}(h)$ is not a sum of squares, then there are also such points in the interior of the hyperbolicity cone as the cone of sums of squares is closed.
\end{remark}

\begin{remark}
 In \cite{soshyperbolic} Saunderson constructs a hyperbolic cubic in $43$ variables whose \emph{B\'ezout matrix} is not a matrix sum of squares. This is the smallest such example that has been known so far. The top left entry of the B\'ezout matrix is the mixed derivative that we are studying. Thus if the latter is not a sum of squares, then the B\'ezout matrix is not a matrix sum of squares.
\end{remark}

Consider the complete graph $K_4$ on $4$ vertices.
We define the spanning tree polynomial of $K_4$ as the element of $\R[x_e : e
\in E(K_4)]$ given by 
\[
    t_{K_4}(x) \ = \ \sum_{\tau} \prod_{e \in \tau}x_e \, ,
\]
where $\tau \subset E(K_4)$ ranges over all edge sets of spanning trees of
$K_4$. The polynomial $t_{K_4}$ is multiaffine, homogeneous and stable
\cite[Thm.~1.1]{halfplane}.  Let $T$ be its minor lift.  Finally, let $p$ be
the polynomial obtained from $T$ by evaluating $T$ at the matrix of
indeterminants
\[
    A=\bordermatrix{
 &  12 &  13 & 14 & 23 & 24 & 34 \cr
 &x_{1}&0&0&0&0&0\cr
&0&x_{2}&a&b&c&0\cr
&0&a&x_{2}&c&b&0\cr
&0&b&c&x_{2}&a&0\cr
&0&c&b&a&x_{2}&0\cr
&0&0&0&0&0&x_{3}}.
\]
Thus $p$ is hyperbolic with respect to every positive definite matrix that can
be obtained by specializing entries of $A$ to some real numbers.
In particular, the polynomial $$W=\frac{\partial p}{\partial x_{1}}\cdot\frac{\partial p}{\partial x_{3}}-p\cdot\frac{\partial^2 p}{\partial x_{1}\partial x_{3}}$$is nonnegative.
We will show that it is not a sum of squares.
We first study the real zero set of $W$.

\begin{lemma}\label{lem:idealj}
 The polynomial $W$ is contained in the ideals $J_1,J_2,J_3$ and $J_4$ where
 \begin{enumerate}
     \item $J_1$ is generated by all $2\times2$ minors of $A$,
     \item $J_2$ is generated by all off-diagonal entries of $A$,
     \item $J_3$ is generated by $a, c$ and $x_2$,
     \item $J_4$ is generated by $b, c$ and $x_2$.
 \end{enumerate}
 \end{lemma}

\begin{proof}
 Part (1) follows from the fact that both $h$ and $\frac{\partial h}{\partial x_{1}}$ are in $J_1$. The other claims are apparent since \[\frac{1}{4}W=a^2 b^2+a^2 c^2+b^2 c^2+c^4-8 a b c x_2+2 a^2 x_2^2+2 b^2 x_2^2.\qedhere\]
\end{proof}

\begin{defi}
 An ideal $I$ in a ring $A$ is called \Def{real radical} if
    $g_1^2+\cdots+g_r^2\in I$ implies $g_1,\ldots,g_r\in I$ for all
    $g_1,\ldots,g_r\in A$.
\end{defi}

\begin{lemma}\label{lem:realradj}
 The ideal $J=\bigcap_{k=1}^4J_k$ is real radical.
\end{lemma}

\begin{proof}
 It suffices to show that each $J_k$ is a radical ideal such that the real points lie Zariski dense in its zero set. This is clear for $J_2,J_3$ and $J_4$. Using Macaulay2 \cite{M2} we checked that $J_1$ is radical. Moreover, the primary decomposition of $J_1$ shows that the zero set of $J_1$ is a union of linear spaces. This implies that the real zeros of $J_1$ are dense as well.
\end{proof}

\begin{theorem}
 The polynomial $W$ is not a sum of squares.
\end{theorem}

\begin{proof}
    Assume for the sake of a contradiction that $W=g_1^2+\cdots+g_r^2$ for some
    polynomials $g_i$. Since $W\in J$ by Lemma \ref{lem:idealj}, Lemma
    \ref{lem:realradj} implies that each $g_i$ is in $J$. Thus $W$ is even in
    the ideal $J\cdot J$. Using Macaulay2 \cite{M2} one checks that this is
    not the case.
\end{proof}

\begin{remark}
 In the terminology of \cite{soshyperbolic} this shows in particular that $h$ is neither SOS-hyperbolic nor weakly SOS-hyperbolic.
\end{remark}

\section{The Spectral Containment Property}\label{sec:eigenvaluecontainment}
We would like to relate the hyperbolicity cone of a homogeneous stable polynomial with the hyperbolicity cone of its minor lift. Recall from \Cref{def:eigenvalue_containment} that a homogeneous multiaffine stable polynomial $p$ has the \emph{spectral containment property} if for any $X \in H(P)$, there is some vector $\lambda$ consisting of the eigenvalues of $X$ with appropriate multiplicity so that $\lambda \in H(p)$.
Elementary symmetric polynomials have the spectral containment property, and we will show that several other polynomials have the spectral containment property in this section. The remainder of this section is devoted to proving some sufficient conditions for the spectral containment property, as well as showing some connections between this property and the Schur-Horn theorem.
\subsection{Schur--Horn Theorem and stable linear functions}
Recall that a linear homogeneous polynomial $p(x) = a_1x_1 + \dots + a_nx_n$ is stable if and only if either $a_i \ge 0$ for each $i \in [n]$, or $a_i \le 0$ for each $i \in [n]$. Moreover, in this case $H(p) = \{x \in \R^n : p(x) \ge 0\}$. These are the simplest stable polynomials and yet it is not completely trivial to show that they have the spectral containment property.

\begin{theorem}\label{thm:lineareigprop}
    Every stable linear homogeneous polynomial has the spectral containment property. 
\end{theorem}

In order to prove this, we will use Schur's contribution to the Schur--Horn
theorem.
\begin{theorem}[Schur]\label{thm:schur_horn}
    Let $p : \R^n \rightarrow \R$ be a homogeneous linear function, and let
    $P$ be the associated minor lift.  Let $A$ be a symmetric matrix and let
    $\lambda$ be an eigenvalue vector for $A$.  Let $\mathfrak{S}_n$ denote
    the symmetric group which acts on $\R^n$ by permuting coordinates.  Let
    $\textnormal{O}(n)$ denote the orthogonal group of $n\times n$ matrices.
    Then
    \[
        \max_{\pi \in \mathfrak{S}_n} p(\pi(\lambda)) = \max_{U \in \textnormal{O}(n)} P(UAU^{\intercal}).
    \]
\end{theorem}

\begin{proof}[Proof of \Cref{thm:lineareigprop}]
    Suppose that $A \in H(P)$, which is equivalent to $P(A) \ge 0$.  By the
    Schur--Horn theorem, there is some eigenvalue vector of $A$, say
    $\lambda$, so that $p(\lambda) \ge P(A) \ge 0$. Thus, there is an
    eigenvalue vector of $A$ contained in $H(p)$ as desired.
\end{proof}

We will see in \Cref{sec:schurhornprop} that if an appropriate generalization of the Schur-Horn theorem holds, then we would be able to show the spectral containment property for a large class of polynomials.

\subsection{Operations Preserving the Spectral Containment Property}
In this section we prove that the spectral containment property is preserved under some simple operations involving adjoining a new variable.
\begin{lemma}
Let $q\in\R[x_1,\ldots,x_n]$ be stable, multiaffine and homogeneous. Let $p\in\R[x_0,\ldots,x_n]$ be defined by $p(x_0, \dots, x_n) = q(x_1, \dots, x_n)$.
If $q$ has the spectral containment property, then $p$ has the spectral containment property.
\end{lemma}
\begin{proof}
First note that $x = (x_0,\ldots,x_n) \in H(p)$ if and only if $(x_1,\ldots,x_n) \in H(q)$.
Let $X \in H(P)$, then we can divide $X$ into blocks as
\[
X = \begin{pmatrix}
        X_{00} & v^{\intercal}\\
        v & M
\end{pmatrix}.
\]
Here, $M$ is equal to $X|_{[n]}$, and $v$ is some element of $\R^n$.

If $I_m$ is the $n\times n$ identity matrix, we can see from the definition of $P$ that $P(X + tI_{n+1}) = Q(M + tI_{n})$. Therefore, for $t \ge 0$, $Q(M + tI_n) = P(X+tI_{n+1}) \ge 0$, which implies $M \in H(Q)$. Let $\lambda(M)$ and $\lambda(X)$ be eigenvalue vectors of $M$ and $X$ respectively, with the property that the entries of $\lambda(M)$ and $\lambda(X)$ appear in increasing order.
The Cauchy interlacing inequalities say that

\[
    \lambda_{0}(X) \le 
    \lambda_{1}(M) \le 
    \lambda_{1}(X) \le 
    \lambda_{2}(M) \le 
    \lambda_{2}(X) \le \dots \le
    \lambda_{{n}}(M) \le
    \lambda_{{n}}(X).
\]
Thus for $i\in[n]$ we can write $\lambda_i(X)=\lambda_i(M)+\epsilon_i$ for some $\epsilon\geq0$. Since $q$ has the spectral containment property, there is a permutation $\sigma$ such that $(\lambda_{\sigma(i)}(M))_{1\leq i\leq n}\in H(q)$. Since the hyperbolicity cone of the stable polynomial $q$ is convex and contains the nonnegative orthant, we also have $(\lambda_{\sigma(i)}(X))_{1\leq i\leq n}=(\lambda_{\sigma(i)}(M)+\epsilon_{\sigma(i)})_{1\leq i\leq n}\in H(q)$. This implies that $(\lambda_0(X),\lambda_{\sigma(1)}(X),\ldots,\lambda_{\sigma(n)}(X))\in H(p)$.
\end{proof}

The spectral containment property is also preserved when multiplying by a
new variable.

\begin{prop}\label{lem:multiplyx0}
    Let $q\in\R[x_1,\ldots,x_n]$ be stable, multiaffine and homogeneous. Let
    $p\in\R[x_0,\ldots,x_n]$ defined by $p(x_0, \dots, x_n) = x_0q(x_1, \dots,
    x_n)$.  If $q$ has the spectral containment property, then $p$ has the
    spectral containment property.  
\end{prop}
Before we show this, we need another lemma.
Let $X$ be a matrix written in block form as 
\[
    X = \begin{pmatrix} X_{00} &
v^{\intercal}\\ v & M \end{pmatrix} 
\]
and $X_{00} \neq 0$.
We write $X / 0 := M - X_{00}^{-1}vv^{\intercal}$ for the Schur complement.
\begin{lemma}\label{lem:sup}
    Let $q \in \R[x_1, \dots, x_n]$ be stable, multiaffine and homogeneous.
    Let $p = x_0 q \in \R[x_0, \dots, x_n]$, and $X \in H(P)$, with $X_{00} >
    0$, then $X / 0 \in H(Q)$.
\end{lemma}
\begin{proof}
    Note that a vector $x = (x_0,x_1,\dots,x_n) \in H(p)$ if and only if $x_0
    \ge 0$ and $(x_1,\dots,x_n) \in H(q)$.  Recall the determinant formula for
    Schur complements: for any $n\times n $ matrix $X$, 
    \[
        \det(X) = X_{00} \det(X / 0).
    \]
    Also, it is not hard to see from the definition that if $S\subseteq
    \{0,1,\dots, n\}$, and $ 0\in S$, then
    \[
        X|_{S} / 0 = X/0|_{(S\setminus 0)}\, ,
    \]
    that is, Schur complements interact naturally with taking submatrices.
    Therefore,
    \[
        P(X) = \sum_{S \subseteq \{0, \dots, n\}} a_S \det(X|_S) =  \sum_{S
        \subseteq \{0, \dots, n\}} a_S X_{00}\det((X /0)|_{S\setminus 0}) =
        X_{00} Q(X / 0)
    \]
    Thus, if $X \in H(P)$ and $X_{00} > 0$, then
    \[
        Q(X / 0) = \frac{P(X)}{X_{00}} \ge 0
    \]
    We can strengthen this result by noting that if we let $J$ be the block
    diagonal matrix given by 
    \[
        J = \begin{pmatrix}
                0 & 0\\
                0 & I_n,
            \end{pmatrix}
    \]
    then $J \in H(P)$, since it is in particular positive semidefinite.  It is
    clear from the definition that $X / 0 + tI_{n+1} = (X+tJ) / 0$.  Thus, we
    have that for all $t \ge 0$,
    \[
        Q(X / 0 + tI_{n+1}) = Q((X+tJ) / 0) = \frac{P(X + tJ)}{X_{00}} \ge 0,
    \]
which implies that $X / 0 \in H(Q)$.
\end{proof}

\begin{proof}[Proof of \Cref{lem:multiplyx0}]
First assume that $X_{00} > 0$.
By \Cref{lem:sup}, and the spectral containment property for $q$, we have that there is an ordering of the eigenvalues of $X / 0$ so that $\lambda(X / 0) \in H(q)$.

Now, we can write
\[
    X = \begin{pmatrix} 0 & 0 \\ 0 & X / 0 \end{pmatrix} + \begin{pmatrix} X_{00} & v^{\intercal} \\ v & X_{00}^{-1}vv^{\intercal} \end{pmatrix},
\]
where the second term is a rank 1 positive semidefinite matrix.

Let $X' = \begin{pmatrix} 0 & 0 \\ 0 & X / 0 \end{pmatrix}$.
Note that $X'$ is block diagonal, so that if $\lambda(X')$ is an eigenvalue vector for $X/0$, then $0 \oplus \lambda(X')$ is an eigenvalue vector for $X'$.
In particular, by ordering the entries appropriately, $\lambda(X') \in H(p)$, from our characterization of $H(p)$ in terms of $H(q)$.

By the Weyl inequalities, there is an ordering of the eigenvalues of $X$ so that $\lambda_i(X) \ge \lambda_i(X')$ for each $i$. This implies that 
\[
    \lambda(X) = (0\oplus \lambda(X')) + v
\]
where $v$ is a nonnegative vector, and therefore $v$ is in $H(p)$. Therefore, $\lambda \in H(p)$.

The case of $X_{00}=0$ follows from continuity of eigenvalues. Observe that if $X$ is in the interior of $H(P)$, then $X_{00} > 0$, and also, since the eigenvalues of a symmetric matrix vary continuously with the matrix, the property of having an eigenvalue vector in $H(p)$ is closed.
Therefore, since $H(p)$ is closed and has nonempty interior, there is an eigenvalue vector of $X$ in $H(p)$.
\end{proof}

\subsection{Polynomials Interlacing an Elementary Symmetric Polynomial}
The spectral containment property can be proved more easily for polynomials which interlace some elementary symmetric polynomial.

Before stating the main result, we note that the minor lift map preserves interlacing.
\begin{lemma}
Let $p,q\in\R[x_1,\ldots,x_n]$ be stable, multiaffine and homogeneous. Let $P,Q$ be the associated minor lifts. Then $p$ interlaces $q$ if and only if $P$ interlaces $Q$.
\label{lem:interlacingpreservers}
\end{lemma}

\begin{proof}
Assume that $p$ interlaces $q$.
Then by the multivariate Hermite--Biehler Theorem \cite[Thm.~5.3]{MR2353258} we have that $p+iq$ is stable.
Let $A$ be a symmetric $n\times n$ matrix.
We have to show that $P(tI+A)$ interlaces $Q(tI+A)$. From 
\cite[Thm.~1.3]{weyl} we see that the linear operator $T_A$ that sends a multiaffine polynomial $p$ to the polynomial $P(\Diag(x_1, \dots, x_n) + A)$ is a stability preserver.
Thus $T_A(p+iq)$ is stable. Substituting $t$ for all variables in $T_A(p+iq)$ shows that $P(tI+A)+iQ(tI+A)$ is stable. Now the claim follows from another application of the Hermite--Biehler Theorem. The other direction is clear, since $p$ and $q$ are the respective restrictions of $P$ and $Q$ to the diagonal matrices.
\end{proof}

\begin{lemma} \label{lem:interlacing_e_k}
Suppose that $p$ is a stable, multiaffine and homogeneous polynomial of degree $d$, and that $e_{d-1}$ interlaces $p$.
Further suppose that for any $X \in H(P)$, there is some eigenvalue vector
    $\lambda$ of $X$, such that $p(\lambda) \ge P(X)$.
Then $p$ has the spectral containment property.
\end{lemma}
\begin{proof}
    We first note the fact that if $p$ is any hyperbolic polynomial, and $q$ interlaces $p$, then $x$ is in the interior of $H(p)$ if and only if $x$ is in $H(q)$ and $p(x) > 0$. This follows easily from considering the bivariate case.

    Let $X$ be in the interior of $H(P)$.
    We first want to show that there is an eigenvalue vector of $X$ that is contained in $H(p)$; the case for general $X$ will then follow from the fact that the eigenvalues of a symmetric matrix are continuous as a function of the entries of the matrix.

    Since $e_{d-1}$ interlaces $p$, by \Cref{lem:interlacingpreservers}, we have that $E_{d-1}$ interlaces $P$. From this, we conclude that since $X \in H(P)$, $X$ is contained in $H(E_{d-1})$, and so any vector of eigenvalues of $X$ is contained in $H(e_{d-1})$.

    Let $\lambda$ be any eigenvalue vector of $X$ so that $0 < P(X) \le p(\lambda)$, then we see that this $\lambda$ must then be in the interior of $H(p)$, as desired.
\end{proof}

In \Cref{lem:interlacers_open}, we show that the set of stable multiaffine forms interlacing $e_{d-1}$ is an open subset containing $e_{d}$.
This implies that if we have a hyperbolic polynomial $p$ which is sufficiently close to $e_{d}$, then $p$ will have the spectral containment property as long as for any $X \in H(P)$, there is some eigenvalue vector $\lambda$, so that $p(\lambda) \ge P(X)$.

We will apply this lemma in a few cases, together with some variational characterizations for eigenvalues to show the spectral containment property for some special kinds of polynomials.

\begin{lemma}\label{lem:sosmult}
Let $p,q$ be multiaffine polynomials of degree $d+1$ and $d$, and let $a\in\R^n$. There exist multiaffine polynomials $m_1,\ldots, m_s, n_1,\ldots,n_s$ of degree $d$ such that $$\De_ap\cdot q-p\cdot \De_a q=m_1n_1+\ldots+m_sn_s.$$ 
\end{lemma}

\begin{proof}
This is straightforward.
\end{proof}

\begin{prop}\label{lem:interlacers_open}
There is an open neighborhood $U$ of $e_{d+1}$ in the vector space of multiaffine forms of degree $d+1$ such that every stable multiaffine $p\in U$ of degree $d+1$ is interlaced by $e_d$.
\end{prop}

\begin{proof}
Let $I$ be the ideal generated by all multiaffine polynomials of degree $d$ and let $V$ be the degree $2d$ part of $I^2$. Let $\Sigma\subset V$ be the set of all polynomials that can be written as a sum of squares of multiaffine polynomials of degree $d$. It follows from the proof of \cite[Thm.~6.2]{kummer2020spectral} that $\De_e e_{d+1}\cdot e_{d}-e_{d+1}\cdot \De_e e_{d}$ is in the interior of $\Sigma$ (with respect to the euclidean topology on $V$). Thus it follows from \Cref{lem:sosmult} that there is an open neighborhood $U$ of $e_{d+1}$ such that for every stable multiaffine $p\in U$ the polynomial $\De_e p\cdot e_{d}-p\cdot \De_e e_{d}$ is in $\Sigma$. Thus $e_d$ interlaces $p$ by \cite[Thm.~2.1]{interlacers}.
\end{proof}

\subsection{Generalized Schur-Horn Property and the Spectral Containment Property}\label{sec:schurhornprop}
We say that an $n$-variate multiaffine homogeneous polynomial $p$ has the \textbf{Schur-Horn property} if for any $n\times n$ symmetric matrix $X$ with some eigenvalue vector $\lambda$,
\[
    \max_{\pi \in \fS_n} p(\pi(\lambda)) = \max_{U \in O(n)} P(UXU^{\intercal}).
\]
The Schur-Horn property for $p$ is equivalent to the fact that for any $n\times n$ symmetric matrix $X$ with eigenvalue vector $\lambda$,
\[
    \max_{\pi \in \fS_n} p(\pi(\lambda)) \ge P(X).
\]
Another equivalent formulation states that $p$ has the Schur-Horn property if
and only if the maximum of $P(UXU^{\intercal})$ as $U$ varies over $O(n)$ is
obtained for some $U$ such that $UXU^{\intercal}$ is diagonal.

The Schur-Horn theorem states that any linear homogeneous polynomial has the Schur-Horn property. We now relate Schur-Horn property and the spectral containment property.
\begin{theorem}\label{thm:schur_horn_to_eigenvalue}
    Let $p$ is an homogeneous multiaffine form of degree $d$. If $p$ has the
    Schur-Horn property, and $e_{d-1}$ interlaces $p$, then $p$ has the
    spectral containment property.
\end{theorem}
\begin{proof}
    It is clear that if $p$ has the Schur-Horn property, then in particular,
    for any $X \in H(P)$, there is some eigenvalue vector $\lambda$ so that
    $p(\lambda) \ge P(X)$. Therefore, $p$ has the spectral containment
    property by \Cref{lem:interlacing_e_k}.
\end{proof}

Using the Schur-Horn property and our previous lemmas, we can show that a family of stable polynomials have the spectral containment property.
\begin{lemma}\label{lem:schur_horn_add}
    If $p$ is a degree $d$ homogeneous multiaffine polynomial with the Schur-Horn property, then $e_d(x) + p$ also has the Schur-Horn property.
\end{lemma}
\begin{proof}
    It can easily be seen that if $X$ is an $n\times n$ symmetric matrix, with an eigenvalue vector $\lambda$, that
    \begin{align*}
        \max_{\pi \in \fS_n} (e_d(\pi(\lambda)) + p(\pi(\lambda))) &= 
        e_d(\lambda) + \max_{\pi \in \fS_n} p(\pi(\lambda))\\
        & = E_d(X) + \max_{U \in O(n)} P(UXU^{\intercal})\\
        & = \max_{U \in O(n)} E_d(UXU^{\intercal}) + P(UXU^{\intercal}).
    \end{align*}
    This gives the desired result.
\end{proof}

\begin{lemma}
    If $p$ is a degree $d$ homogeneous multiaffine polynomial with the Schur-Horn property, then for $\epsilon > 0$ sufficiently small, $e_d(x) + \epsilon p$ has the spectral containment property.
\end{lemma}
\begin{proof}
    By \Cref{lem:interlacers_open}, we see that for $\epsilon$ sufficiently small, $e_d(x) + \epsilon p$ is interlaced by $e_{d-1}$.
    Moreover, by \Cref{lem:schur_horn_add}, we see that $e_d(x) + \epsilon p$ has the Schur-Horn property.
    Therefore, by \Cref{thm:schur_horn_to_eigenvalue}, we see that $e_d(x) + \epsilon p$ has the spectral containment property.
\end{proof}
We now give some examples of polynomials with the Schur-Horn property.

\subsection{The Schur-Horn Property For Degree $n-1$ Polynomials}

\begin{theorem}\label{thm:inverseSchurHorn}
    If $p \in \R[x_1, \dots, x_n]$ is a degree $n-1$ multilinear homogeneous polynomial, then $p$ has the Schur-Horn property.
\end{theorem}
\begin{proof}
    Write $p(x) = \sum_{i=1}^n a_i \prod_{j\in [n] \setminus i}x_i$.
    In this case, 
    \[
        P(X) = \sum_{i=1}^n a_i\det(X|_{[n] \setminus i})
    \]

    Recall that the dual of $p(x)$ was defined in \Cref{sec:minor_lift}, as 
    \[
        p^*(x) = \sum_{i=1}^n a_i x_i.
    \]
    Abusing notation, we define $P^*$ to be 
    \[
        P^*(X) = \sum_{i=1}^n a_i X_{ii}.
    \]

    Define the adjugate matrix of $X$ by $\Adj(X) = \det(X) X^{-1}$.
    By Cramer's rule, the diagonal entries of the adjugate matrix are given by
    \[\Adj(X)_{ii} = \det(X|_{[n] \setminus i}).\]
Hence, using \Cref{rmk:minor_lift_dual}, we see that $P^*(\Adj(X)) = P(X)$.

    The eigenvalues of $\Adj(X)$ are of the form $\mu_j = \prod_{i \in [n] \setminus j} \lambda_i$ where $\lambda$ is an eigenvalue vector of $X$. 
    We see then that $p^*(\mu) = p(\lambda)$. Now we apply the Schur-Horn theorem to the linear form $p^*$ and the matrix $\Adj(X)$ to see that 
    \[
        \max_{\pi \in \fS_n} p^*(\pi(\mu)) = \max_{U \in O(n)} P^*(U^{\intercal}\Adj(X)U).
    \]

    Applying our identities relating the vector $\mu$ to $\lambda$, we see that
    \[
        \max_{\pi \in \fS_n} p(\pi(\lambda)) = \max_{U \in O(n)} P(U^{\intercal}XU).
    \]
\end{proof}

From this, we immediately obtain a corollary.
\begin{cor}
    There is an open set $U$ in the space of degree $n-1$ homogeneous multiaffine polynomials, such that $U$ contains $e_{n-1}$ and every element of $U$ is stable and has the spectral containment property.
\end{cor}

\subsection{Extensions of Elementary Symmetric Polynomials and the Schur--Horn Property}
We note that it is unclear whether the Schur--Horn property is preserved by adding extra variables. We show that this holds for elementary symmetric polynomials.
\begin{prop}\label{prop:SH}
    Fix natural numbers $d \le k \le n$.
    Let $p = \pm e_d(x_1, \dots, x_k) \in \R[x_1, \dots, x_n]$. Then, $p$ has the Schur-Horn property.
\end{prop}
    We can reduce this to the classical Schur-Horn theorem. To do this, we require a lemma involving a construction, which is referred to in \cite[Chapter 3]{marvin1973finite} as a derivation of a matrix $X$. 
    \begin{lemma}\label{lem:derivation}
        For any $n\times n$ symmetric matrix $X$, with eigenvalue vector $\lambda$, there exists a $\binom{n}{k}\times \binom{n}{k}$ symmetric matrix $D^{k,d} X$ with the following two properties: 
        \begin{itemize}
            \item The eigenvalues of $D^{k,d} X$ are precisely those real numbers of the form
                \[e_d(\lambda_{s_1}, \lambda_{s_2}, \dots, \lambda_{s_k}),\]
                where we range over all possible values of $s_1, \dots, s_k\in [n]$ so that $s_1< s_2< \dots< s_k$.
            \item The diagonal entries of $D^{k,d} X$ are precisely those of the form $E_d(X|_{S})$, where $S$ ranges over the size $k$ subsets of $[n]$.
        \end{itemize}
    \end{lemma}\begin{proof}[Proof of \Cref{lem:derivation}]
    We will define $D^{k,d}X$ in terms of wedge powers. If we regard $X$ as an endomorphism from $\R^n$ to $\R^n$, then $D^{k,d}X$ is defined as an endomorphism of $\wedge^k \R^n$ by letting 
    \[
        D^{k,d}X(v_1\wedge v_2 \wedge \dots \wedge v_k) = \sum_{S \in \binom{[n]}{k}} w_{S,1}\wedge w_{S,1} \wedge \dots \wedge w_{S,k},
    \]
    where
    \[
        w_{S, k} = \begin{cases} Xv_k \text{ if }k\in S\\ v_k \text{ if }k \not \in S\end{cases}.
    \]

    It is not hard to see that if $v_1, \dots, v_k$ are linearly independent eigenvectors of $X$ with eigenvalues $\lambda_1, \dots, \lambda_k$ respectively, then $v_1 \wedge \dots \wedge v_k$ is an eigenvector of $D^{k,d}X$ with eigenvalue $e_d(\lambda_1, \dots, \lambda_k)$, and this clearly implies the first property in \Cref{lem:derivation}.

    On the other hand, if we use the natural basis of $\wedge^k \R^n$ given by
        $\{e_{s_1} \wedge e_{s_2} \wedge \dots \wedge e_{s_k}\}$, where $e_i$ is a standard basis vector, and $s_1 < s_2 <\dots <s_k$, then this basis is orthogonal under the natural inner product of $\wedge^k \R^n$, and also
    \[
        (e_{s_1} \wedge e_{s_2} \wedge \dots
        \wedge e_{s_k})^{\intercal}D^{k,d}X(e_{s_1} \wedge e_{s_2} \wedge \dots
      \wedge  e_{s_k}) = E_{d}(X|_{\{s_1, \dots, s_k\}}).
    \]
    This clearly implies the second property of \Cref{lem:derivation}.
\end{proof}

\begin{proof}[Proof of \Cref{prop:SH}]

   The classical Schur-Horn theorem implies that for any symmetric matrix $X$,
   \[
       \max_{s_1 < s_2 < \dots < s_k} e_d(\lambda_{s_1}, \lambda_{s_2}, \dots, \lambda_{s_k}) \ge 
       \max_{S \in \binom{[n]}{k}} E_d(X|_{S}) \ge  E_d(X|_{1,\dots,k}), 
    \]
    and also that
   \[
       \min_{s_1 < s_2 < \dots < s_k} e_d(\lambda_{s_1}, \lambda_{s_2}, \dots, \lambda_{s_k}) \le 
       \min_{S \in \binom{[n]}{k}} E_d(X|_{S}) \le  E_d(X|_{1,\dots,k}).
    \]

    This first statement implies the Schur-Horn property for $e_d(x_1, \dots, x_k)$, and the second implies the Schur-Horn property for $-e_d(x_, \dots, x_k)$.
    \end{proof}

\subsection{A Small Example of the Schur-Horn Property}

We give one more example of the Schur-Horn property, which is noteworthy because our proof does not appeal to the classical Schur-Horn theorem.
\begin{lemma}
    The polynomial $x_1(x_2+x_3) \in \R[x_1, x_2, x_3, x_4]$ has the Schur-Horn property.
\end{lemma}
\begin{remark}
    The polynomial $x_1(x_2+x_3) \in \R[x_1, x_2, x_3]$ clearly has the Schur-Horn property, by \Cref{thm:inverseSchurHorn}, but it is not clear that this remains the case if we introduce a new variable.
\end{remark}
\begin{proof}
    Let $D = \Diag(\lambda_1, \lambda_2, \lambda_3, \lambda_4)$. Let $U$ be in $\SO(4)$, and write its columns as 
    \[
        U = \begin{pmatrix} v & w & z & y \end{pmatrix}.
    \]

    It is not hard to see via an explicit computation that
    \[
        P(UDU^{\intercal}) = \det
    \begin{pmatrix}
    \sum_{i=1}^4 \lambda_i v_i^2 & \sum_{i=1}^4 \lambda_i w_iv_i\\
    \sum_{i=1}^4 \lambda_i w_iv_i & \sum_{i=1}^4 \lambda_i w_i^2\\
    \end{pmatrix}
+ \det
    \begin{pmatrix}
    \sum_{i=1}^4 \lambda_i v_i^2 & \sum_{i=1}^4 \lambda_i z_iv_i\\
    \sum_{i=1}^4 \lambda_i z_iv_i & \sum_{i=1}^4 \lambda_i z_i^2\\
    \end{pmatrix}.\\
    \]
We expand this formulation out by multilinearity of the determinant to obtain the following
\begin{align}
\sum_{i=1}^4 \sum_{j=1}^4 \lambda_i \lambda_j
    \left(\det
    \begin{pmatrix}
    v_i^2 & w_jv_j\\
    w_iv_i & w_j^2\\
    \end{pmatrix}
+ \det
    \begin{pmatrix}
    v_i^2 & z_jv_j\\
    z_iv_i & z_j^2\\
    \end{pmatrix}\right)\label{eq:multilinear}\\
= \sum_{i=1}^4 \sum_{j<i} \lambda_i \lambda_j
    \left(\det
    \begin{pmatrix}
    v_i^2 & w_jv_j\\
    w_iv_i & w_j^2\\
    \end{pmatrix}
+ \det
    \begin{pmatrix}
    v_i^2 & z_jv_j\\
    z_iv_i & z_j^2\\
    \end{pmatrix}
+ \det
    \begin{pmatrix}
    v_j^2 & w_iv_i\\
    w_jv_j & w_i^2\\
    \end{pmatrix}
+ \det
    \begin{pmatrix}
    v_j^2 & z_iv_i\\
    z_jv_j & z_i^2\\
    \end{pmatrix}\right)\label{eq:weird}\\
= \sum_{i=1}^4 \sum_{j<i} \lambda_i \lambda_j
    \left(\det
    \begin{pmatrix}
    v_i & w_i\\
    v_j & w_j\\
    \end{pmatrix}^2
+ \det
    \begin{pmatrix}
    v_i & z_i\\
    v_j & z_j\\
    \end{pmatrix}^2 \right).
\end{align}

We can think of this as a polynomial 
\[
    \gamma(\lambda) = \sum_{i = 1}^4 \sum_{j < i}\gamma_{i,j} \lambda_i \lambda_j, 
\]
where
\[
    \gamma_{i,j} = \det
    \begin{pmatrix}
    v_i & w_i\\
    v_j & w_j\\
    \end{pmatrix}^2
+ \det
    \begin{pmatrix}
    v_i & z_i\\
    v_j & z_j\\
    \end{pmatrix}^2.
\]
We make the following claim:
\begin{lemma}
\label{lmma:convex_hull}
$\gamma(\lambda)$ is in the convex hull of the polynomials 
\[
    a_{i,j,k}(\lambda) = \lambda_i(\lambda_j + \lambda_k)
\]
where $\{i, j, k\} \in \binom{[4]}{3}$.
\end{lemma}
To see that  \Cref{lmma:convex_hull} implies the theorem, observe that for any $i,j,k$, 
\[
a_{i,j,k}(\lambda) \ge \max_{\pi \in S_4} p(\lambda_{\pi(1)},\lambda_{\pi(2)} ,\lambda_{\pi(3)} ,\lambda_{\pi(4)}).
\]
Hence, in particular, any convex combination of the $a_{i,j,k}$ will be lower bounded by this same quantity.

Therefore, since every symmetric matrix is diagonalizable, we have that for any symmetric matrix $X$, and any orthogonal $U$,
\[
    P(U^{\intercal}XU) \ge \max_{\pi \in S_4} p(\lambda_{\pi(1)},\lambda_{\pi(2)} ,\lambda_{\pi(3)} ,\lambda_{\pi(4)}).
\]
The opposite inequality is easy to see by choosing $U$ to be an orthogonal matrix diagonalizing $X$.
\end{proof}

It remains to show \Cref{lmma:convex_hull}.

\begin{proof}[Proof of \Cref{lmma:convex_hull}]

Since we work in dimension $4$, we can find the inequalities defining this convex hull explicitly using computational methods.
It turns out that the polynomial $\gamma$ is in the convex hull of the $a_{i,j,k}$ if and only if for each $i,j \in [n]$,
\[
\gamma_{i,j} \ge 0 
\]
and if $\{i,j,k,l\} = \{1,2,3,4\}$, then
\[
\gamma_{i,j} + \gamma_{k,l} \le 1
\]
and
\[
\sum_{i,j \in \binom{[n]}{2}} \gamma_{i,j}  = 2
\]

For our particular value of $\gamma_{i,j}$, it is easy to see that it is the sum of two squares and hence nonnegative. Further notice that the sum of all of the coefficients of $\gamma$ is $\gamma(1,1,1,1)$, so that returning to the definition of the polynomial $\gamma$,
\[
    \gamma(1,1,1,1) = P(I) = 2.
\]

It remains to show that 
\[
\gamma_{i,j} + \gamma_{k,l} \le 1
\]
Notice that
\[
\gamma_{i,j} + \gamma_{k,l} = 
   \det \begin{pmatrix}
    v_i & w_i\\
    v_j & w_j\\
    \end{pmatrix}^2
+ \det
    \begin{pmatrix}
    v_i & z_i\\
    v_j & z_j\\
    \end{pmatrix}^2
+ \begin{pmatrix}
    v_k & w_k\\
    v_l & w_l\\
    \end{pmatrix}^2
+ \det
    \begin{pmatrix}
    v_k & z_k\\
    v_l & z_l\\
    \end{pmatrix}^2
\]
We prove that this is at most 1. Recall that
\[
U = 
\begin{pmatrix}
v & w & z & y
\end{pmatrix}
=
\begin{pmatrix}
v_1 & w_1 & z_1 & y_1\\
v_2 & w_2 & z_2 & y_2\\
v_3 & w_3 & z_3 & y_3\\
v_4 & w_4 & z_4 & y_4\\
\end{pmatrix}.
\]
The Jacobi complementary minors theorem for matrix inverses implies that if $S \subseteq [4]$, then
\[
    \det(U|_{S, T}) = \det(U)\det(U^{-1}|_{S^c, T^c}) = \pm\det(U^{\intercal}|_{S^c, T^c}) = \pm\det(U^{\intercal}|_{T^c, S^c})
\]
We now see that 
\[
\det
\begin{pmatrix}
v_i & w_i\\
v_j & w_k
\end{pmatrix}^2=
\det
\begin{pmatrix}
z_k & y_k\\
z_l & y_l
\end{pmatrix}^2.
\]
Similarly, we must have
\[
\det
\begin{pmatrix}
v_i & z_i\\
v_j & z_j
\end{pmatrix}^2=
\det
\begin{pmatrix}
w_k & y_k\\
w_l & y_l
\end{pmatrix}^2.
\]
Let
\[
M =
\begin{pmatrix}
v_3 & w_3 & z_3 & y_3\\
v_4 & w_4 & z_4 & y_4\\
\end{pmatrix}.
\]
Notice that since $U$ is orthogonal, these two rows are orthogonal, and so $M
M^{\intercal}= I$.  Applying the Cauchy--Binet theorem to
$\det(MM^{\intercal}) = \det(I)$ we see that
\begin{align*}
    1 &= \det(MM^{\intercal}) \\
      &= \sum_{S \subseteq \binom{[n]}{2}} \det(M_{S})\det(M_{S}^{\intercal})\\
      &= \sum_{S \subseteq \binom{[n]}{2}} \det(M_{S})^2\\
    & \ge \begin{pmatrix}
    z_k & y_k\\
    z_l & y_l
    \end{pmatrix}^2
    + \det \begin{pmatrix}
    v_i & w_i\\
    v_j & w_j\\
    \end{pmatrix}^2
+ \begin{pmatrix}
    v_k & w_k\\
    v_l & w_l\\
    \end{pmatrix}^2
+ \det
    \begin{pmatrix}
    v_k & z_k\\
    v_l & z_l\\
    \end{pmatrix}^2\\
    &= \gamma_{i,j} + \gamma_{k,l}  \, .
\end{align*}
The result now follows.
\end{proof}

\section{The permutation property}
The goal of this section is to prove \Cref{thm:permutation}. It says that given any point $v$ in the hyperbolicity cone of $e_k$ and any other homogeneous stable multiaffine polynomial $h$ of the same degree, some permuation of the coordinates of $v$ is in the hyperbolicity cone of $h$. We call this remarkable propery of $e_k$ the \emph{permutation property}. We first need some preparation.

\begin{lemma}
 Assume that the homogeneous stable polynomials $g,h\in\R[x_1,\ldots,x_n]$ have nonnegative coefficients and a common interlacer. Then $f=g+h$ is stable. If $v$ is in the hyperbolicity cone of $f$, then $v$ is in the hyperbolicity cone of $g$ or in the hyperbolicity cone of $h$.
\end{lemma}

\begin{proof}
 Let $e$ be the all-ones vector.
  The univariate polynomials $F=f(te-v), G=g(te-v)$ and $H=h(te-v)$ have a common interlacer. Further, all roots of $F$ are nonnegative. The existence of a common interlacer implies that $G$ and $H$ have at most one negative root each. Assume for the sake of a contradiction that both $G$ and $H$ have a negative root. Then $G$ and $H$ have the same (nonzero) sign on the smallest root of $F$. This contradicts $F=G+H$. Thus either $G$ or $H$ have only nonnegative roots which implies the claim.
\end{proof}

\begin{lemma}
 Let $h\in\R[x_1,\ldots,x_n]$ be homogeneous, multiaffine and stable. Let $\tau\in\fS_n$ be a transposition. Then $h$ and $\tau(h)$ have a common interlacer.
\end{lemma}

\begin{proof}
  Without loss of generality assume that $\tau=(12)$ and let $g=\tau(h)$. We can write $$h=A\cdot x_1\cdot x_2+B\cdot x_1+C\cdot x_2+D$$for some multiaffine $A,B,C,D\in\R[x_3,\ldots,x_n]$. Then the polynomial $$\left(\frac{\partial}{\partial x_1}+\frac{\partial}{\partial x_2}\right)h=A\cdot(x_1+x_2)+B+C=\left(\frac{\partial}{\partial x_1}+\frac{\partial}{\partial x_2}\right)g$$is a common interlacer of $h$ and $g$.
\end{proof}

\begin{cor}\label{cor:transcone}
 Let $h\in\R[x_1,\ldots,x_n]$ be homogeneous, multiaffine and stable. Let $\tau\in\fS_n$ be a transposition, $g=\tau(h)$ and $f=\lambda g+\mu h$ for some nonnegative $\lambda, \mu\in\R$. Then $\textnormal{C}(f,e)\subset\textnormal{C}(g,e)\cup\textnormal{C}(h,e)$.
\end{cor}

\begin{proof}
 This is a direct consequence of the two preceding lemmas.
\end{proof}

Let $\Q[\fS_n]$ be the group algebra of the symmetric group $\fS_n$ on $n$ elements, i.e. $\Q[\fS_n]$ is the vector space over $\Q$ with basis $e_g$ for $g\in\fS_n$ whose ring structure is defined by extending $e_g\cdot e_h:=e_{g\cdot h}$ linearly. In $\Q[\fS_n]$ we have the identity \begin{equation}\label{eq:factor}
    \prod_{j=2}^n\prod_{i=1}^{j-1} \left(1+\frac{1}{j-i}\cdot e_{(ij)}\right)=\sum_{g\in\fS_n}e_g,
\end{equation} see for example \cite[p.~192]{nazarov}. From this we obtain our desired theorem.

\begin{theorem}
 Let $\sigma_d\in\R[x_1,\ldots,x_n]$ be the elementary symmetric polynomial of degree $d$ and $h\in\R[x_1,\ldots,x_n]$ any other nonzero homogeneous multiaffine stable polynomial of degree $d$. If $v$ is in the hyperbolicity cone of $e_d$, then $\tau(v)$ is in the hyperbolicity cone of $h$ for some permutation $\tau\in\fS_n$.
\end{theorem}

\begin{proof}
 We have $c\cdot\sigma_d=(\sum_{g\in\fS_n}e_g)h$ for some nonzero scalar $c\in\R$. Thus by \Cref{eq:factor} we can write $$c\cdot\sigma_d=\left(\prod_{i=1}^r(1+\lambda_ie_{\tau_i})\right)h$$ for some positive $\lambda_i\in\R$, transpositions $\tau_i\in\fS_n$ and $r=\binom{n}{2}$. We define $h_k=\left(\prod_{i=1}^k(1+\lambda_ie_{\tau_i})\right)h$ for $k=0,\ldots,r$. Since $h_k=h_{k-1}+\lambda_k\tau_k(h_{k-1})$, \Cref{cor:transcone} implies that if $v$ is in the hyperbolicity cone of $h_k$, then either $v$ or $\tau_k(v)$ is in the hyperbolicity cone of $h_{k-1}$. Since $h_r=c\cdot\sigma_d$ and $h_0=h$, this argument shows that if $v$ is in the hyperbolicity cone of $\sigma_d$, then $(\tau_{i_1}\circ\cdots\circ\tau_{i_s})(v)$ is in the hyperbolicity cone of $h$ for some $1\leq i_1<\cdots<i_s\leq r$.
\end{proof}
\section{Open problems}\label{sec:openproblems}
Our work sparks a wide range of open problems. We mention some of them here.  For several of these problems, we presented proofs for some special cases, whereas the general case remains open. 

\subsection{Hyperbolic Schur-Horn Theorem}

In \Cref{sec:hadamard-fischer} we proved the hyperbolic generalization of Hadamard-Fischer inequality as well as Koteljanskii's inequality, in \Cref{thm:hadamard-fischer} and \Cref{thm:kota}. Here we present another potential generalization of classical linear algebra results in Schur-Horn theorem. 

Schur-Horn theorem appears in our previous section on Spectral containment Property, where it plays a major role and a generalized version called Schur-Horn property was formed. Here we will form a different generalization of Schur-Horn theorem in terms of hyperbolic polynomials. 

We will formulate our generalization in the language of majorization. Given polynomials $p$ and $q$ of the same degree, both hyperbolic with respect to the direction $v$, we say that $p$ majorizes $q$ in direction $v$ if for all $x \in \R^n$, the roots of $p(x-tv)$ majorize the roots of $q(x-tv)$. Recall that given $\alpha,\beta\in\R^k$, $\alpha$ majorizes $\beta$ if $\sum_{i=1}^k \alpha_i=\sum_{i=1}^k \beta_i$ and the following holds: let $\alpha',\beta'$ be obtained from $\alpha,\beta$ by reordering coordinates such that $\alpha'_1\ge...\ge\alpha'_k$ and $\beta'_1\ge...\ge\beta'_k$, then for each $1\le m<k,\sum_{i=1}^m \alpha'_i\ge\sum_{i=1}^m \beta'_i$. Equivalently, $\alpha$ majorizes $\beta$ if and only if $\beta\in\conv(\fS_k(\alpha))$, where the symmetric group $\fS_k$ acts on $\alpha$ by permuting its coordinates.


In this language, we can restate the Schur direction of the Schur-Horn theorem as follows:
\begin{lemma}(Schur)
	$\det(X)$ majorizes $\det(\diag(X))$ in the identity direction.
\end{lemma}

We conjecture that this holds for all homogeneous PSD-stable lpm-polynomials. 

\begin{conj}
	Let $P$ be a homogeneous PSD-stable lpm-polynomial. Then $P(X)$ majorizes $P(\diag(X))$ in the identity direction. 
\end{conj}

Recall for $1\le k\le n$ we defined $E_k(X)=\sum_{|S|=k}\det X_S$ to be the minor lift of degree $k$ elementary symmetric polynomial, i.e., sum of all $k\times k$ principal minors of $X$. We are able to prove this conjecture for rescalings of $E_k$. Our proof will use the following result from \cite{BB10majorizationpreserver}, which follows from Theorem 1 of their paper. 

\begin{prop}\label{prop:derivative-maj}
	Suppose $p$ majorizes $q$ in direction $v$, then $D_v p$ majorizes $D_v q$, where $D_v$ denotes the directional derivatives in the $v$ direction.
\end{prop}

Now we are ready to state and prove our result. 

\begin{prop}
	Let $D$ be any positive diagonal matrix, and $P(X)=E_k(D^{-1/2}XD^{-1/2})$. Then $P(X)$ majorizes $P(\diag(X))$ in the identity direction. 
\end{prop}

\begin{proof}
	First notice that $\det(X)$ majorizes $\det(\diag(X))$ in the $D$ direction, i.e., roots of $\det(X-tD)$ majorize roots of $\det(\diag(X)-tD)$ for any $X$. This follows from applying original Schur's theorem to the symmetric matrix $D^{-1/2}XD^{-1/2}$, since we have $\det(X-tD)=\det(D)\det(D^{-1/2}XD^{-1/2}-tI)$ and similarly $\det(\diag(X)-tD)=\det(D)\det(D^{-1/2}\diag(X)D^{-1/2}-tI)$. Also notice that $D^{-1/2}\diag(X)D^{-1/2}=\diag(D^{-1/2}XD^{-1/2})$. 
	
	Now we apply \Cref{prop:derivative-maj} $(n-k)$ times, where $p=\det(X),q=\det(\diag(X)),v=D$. This shows that $p^{(k)}(X)=\sum_{|S|=k}\det(X_S)\prod_{i\notin S}D_{ii}$ majorizes $q^{(k)}(X)=\sum_{|S|=k}\det(\diag(X)_S)\prod_{i\notin S}D_{ii}$ in the $D$ direction. Computing $p^{(k)}(X-tD)$ we have
	\begin{align*}
	p^{(k)}(X-tD)&=\sum_{|S|=k}\det(X_S-t D_S)\prod_{i\notin S}D_{ii}\\
	&=\sum_{|S|=k}\det (D_S)\det(D_S^{-1/2}X_S D_S^{-1/2}-t I_S)\prod_{i\notin S}D_{ii}\\
	&=\det(D)\sum_{|S|=k}\det(D_S^{-1/2}X_S D_S^{-1/2}-t I_S)\\
	&=\det(D)E_k(D^{-1/2}XD_{-1/2}-tI)
	\end{align*}
	
	Similarly, $q^{(k)}(X-tD)=\det(D)E_k(D^{-1/2}\diag(X)D_{-1/2}-tI)$. This shows that roots of $E_k(D^{-1/2}XD_{-1/2}-tI)$ majorize roots of $E_k(D^{-1/2}\diag(X)D_{-1/2}-tI)$. This completes the proof. 
\end{proof}

	We may also formulate a hyperbolic generalization of Horn's theorem, which we conjecture to be true but do not have any results.
	
	\begin{conj}
		Let $P$ be any degree $k$ lpm-polynomial. Let $\lambda,\mu\in \R^k$ such that $\lambda$ majorizes $\mu$. Then there exists a symmetric matrix $X$ such that roots of $P(X-tI)$ are given by $\lambda$, and roots of $P(\diag(X)-tI)$ are given by $\mu$.  
	\end{conj}

\subsection{Spectral containment property and the Schur-Horn property}
We showed that many polynomials have the spectral containment property. Based on these examples and additional computational evidence we conjecture the following:

\begin{conj}
    All homogeneous multiaffine stable polynomials have the spectral containment property.
\end{conj}
There are several special cases of this conjecture which are of particular interest, which we enumerate separately.
\begin{conj}
    All quadratic homogeneous multiaffine stable polynomials have the spectral containment property.
\end{conj}
This case is of special interest because quadratic multiaffine polynomials have especially simple minor lifts. Namely, if 
\[
	p(x) = \sum_{i \neq j}a_{ij}x_ix_j,
\]
then
\[
	P(X) = p(\diag(X)) - \sum_{i \neq j}a_{ij}X_{ij}^2.
\]
It is therefore plausible that this conjecture could be proved (or disproved) by exploiting this special structure.
\begin{conj}\label{conj:diagonal_rescaling}
	Let $D$ be a positive definite diagonal matrix, and let $p(x) = e_k(Dx)$. Then $p(x)$ has the spectral containment property.
\end{conj}
Again, this is of special interest because of its relation to diagonal congruence as we now explain. 
\begin{lemma}
	Let $p$ be a homogeneous, multiaffine stable polynomial, let $D$ be a positive definite diagonal matrix, and let $q = p(Dx)$. Then $x \in H(q)$ if and only if $Dx \in H(p)$, and $X \in H(Q)$ if and only if $DXD \in H(P)$.
\end{lemma}
\begin{proof}
	$x \in H(q)$ if and only if $q(x + t\vec{1}) \ge 0$ for all $t \ge 0$. This is equivalent to the statement that $p(D(x + t\vec{1})) = p(Dx + t\diag(D)) \ge 0$ for all $t \ge 0$. Notice though that if $D$ is positive definite, then $\diag(D)$ is in the interior of the hyperbolicity cone of $p$.
	Therefore, $p(Dx + t\diag(D)) \ge 0$ for all $t \ge 0$ if and only if $Dx \in H(p)$.

	Similarly, if $p(x) = \sum_{S\subseteq [n]} a_S \prod_{i \in S}x_i$, we see that $q(x)  = \sum_{S\subseteq [n]} (\prod_{i \in S} D_{ii}a_S) \prod_{i \in S}x_i$. Therefore,
	\[
		Q(X) = \sum_{S \subseteq [n]} (\prod_{i\in S}D_{ii}a_S) \det(X|_S) =  \sum_{S \subseteq [n]} a_S \det((D^{1/2}XD^{1/2})|_S) = P(D^{1/2}XD^{1/2}).
	\]
	We thus have that
	\[
		Q(X+tI) = P(D^{1/2}(X+tI)D^{1/2}) = P(D^{1/2}XD^{1/2} + tD).
	\]
	Because $D$ is positive definite, it is in the interior of $H(P)$, and therefore, $P(D^{1/2}XD^{1/2} + tD) \ge 0$ for all $t \ge 0$ if and only if $D^{1/2}XD^{1/2} \in H(P)$. This implies the result.
\end{proof}
From, this we see that \Cref{conj:diagonal_rescaling} is equivalent to the statement that for any $X \in H(E_k)$, and any positive definite diagonal matrix $D$, we have that there exists an eigenvalue vector $\lambda$ of $D^{1/2}XD^{1/2}$ so that $D^{-1}\lambda \in H(e_k)$.
This gives us a very quantitative relationship between the eigenvalues of a symmetric matrix $X$ and those of $D^{1/2}XD^{1/2}$, which are of fundamental interest in a number of situations.

The Schur-Horn property is another interesting property of a multiaffine polynomial. Once again, despite computer search, we are unable to find an example of a multiaffine homogeneous polynomial that does not have the Schur-Horn property.
From this, we conjecture
\begin{conj}
   All homogeneous multiaffine polynomials have the Schur-Horn property.
\end{conj}


\bibliographystyle{plain}
\bibliography{biblio.bib}
\end{document}